\providecommand{\href}[2]{#2}
\documentclass[a4paper, 12pt]{amsart}
\usepackage{graphicx, subfigure, color}
\usepackage{amsmath, amsfonts, amssymb, latexsym, euscript}
\usepackage{mathptmx, wrapfig, url}
\theoremstyle{plain}
\newtheorem{Thm}{Theorem}

\newtheorem{Ex}[Thm]{Example}
\newtheorem{Coro}[Thm]{Corollary}
\newtheorem{Lem}[Thm]{Lemma}

\newtheorem{Prop}[Thm]{Proposition}

\newtheorem{Def}[Thm]{Definition}
\newtheorem{Rem}[Thm]{Remark}
\usepackage{graphicx, subfigure, color}

\begin{document}

\title{Compression Maps Between Polytopes}

\author{Jos\'{e} Ayala}

\address{Universidad de Tarapacá, Casilla 7D, Iquique, Chile} 
              \email{jayalhoff@gmail.com}
\author{David Kirszenblat}
\address{School of Mathematics and Statistics, University of Melbourne
              Parkville, VIC 3010 Australia}
\email{d.kirszenblat@gmail.com}
\author{J. Hyam Rubinstein}
\address{School of Mathematics and Statistics, University of Melbourne
              Parkville, VIC 3010 Australia}
\email{joachim@unimelb.edu.au}
\maketitle

\begin{abstract} 
A shape of a combinatorial polytope is a convex embedding into Euclidean space. We provide necessary and sufficient conditions for a piecewise linear map between two shapes of the same polytope to be a compression (respectively a weak compression), meaning a distance decreasing (respectively a distance non-increasing) map between distinct pairs of points. We establish that there is a partial order on the space of shapes given by the relation of having a weak compression map between pairs of shapes.

Finally, we construct a compression metric on the projective shape space of a polytope; the space of convex Euclidean realizations modulo rigid motions and homothety. For the projective space shape of a simplex, we show that the compression metric is complete. For general polytopes, we establish that the projective shape space has a natural completion given by the projective shape space of weakly convex realizations. 
\end{abstract}

\section{Introduction}

 A compression (respectively weak compression) is a map between two realizations of a combinatorial polytope as a convex set, where the map is linear on each simplex of a triangulation of the polytope (or piecewise linear on the polytope) and distance decreasing (respectively non-increasing) between arbitrary distinct pairs of points. 
 
 For convenience, if the polytope is a simplex, then the triangulation we use is given by the simplex structure. For general polytopes, the triangulation we mainly consider is the barycentric one. This has the advantage of being independent of the choice of convex realization for the polytope. It would be possible to use other triangulations of polytopes to construct compression maps and we discuss this briefly. 

We denote by $\mathcal{P}$ the space of all realizations of a given combinatorial type of a polytope as a convex set in $\mathbb{R}^n$. Then, we observe that weak compression induces a partial order relation on $\mathcal{P}$.

A natural question is whether a piecewise linear map between two shapes of a polytope, which induces a compression on the set of edges, is automatically a compression map. We provide a simple counterexample, showing that this does not hold even for linear maps between triangles. 

In the special case of a polytope with simplicial facets that can be triangulated by successive truncations at vertices, a compression can be constructed as an orthogonal projection from a pleated convex embedding in a higher-dimensional space to a polytope of the same combinatorial type, which has a convex embedding in a lower-dimensional space. We show that all compressions and weak compressions for this special class $\mathcal{C}$ of polytopes can be constructed in this way. For simplices, it suffices to take convex embeddings into a higher-dimensional space to realize all compression and weak compression maps through orthogonal projections. 

For general polytopes, we give a necessary and sufficient condition for a piecewise linear map between two shapes to be a compression or a weak compression, using the barycentric triangulation. 

We prove that an embedding of a simplex or a polytope in $\mathcal{C}$ as a convex set or pleated convex set respectively into a higher-dimensional space, can be achieved through a succession of embeddings, with the dimension increasing by one at each step. This result may be useful in developing efficient algorithms for constructing compression maps.

Finally, we introduce a compression metric on the projective shape space $\mathcal{P} / \sim$, which parametrizes polytope shapes up to homothety. This metric is obtained from the eigenvalues of the symmetric matrix associated with the affine transformation mapping one shape of a simplex to another. The compression distance between shapes of simplices is defined as the logarithm of the ratio of the maximum and minimum eigenvalues of the symmetric matrix, capturing the anisotropy of the transformation. This formulation ensures scale invariance and provides a natural metric structure on the projective shape space of a simplex. For a polytope, the distance between two shapes is defined as the maximum distance between the corresponding simplices in the barycentric subdivisions of the shapes. 

We show that the compression metric for simplices is complete and for polytopes, the completion is obtained by adding in weakly convex shapes of the polytope. In particular, the dihedral angle between adjacent facets can achieve $\pi$ for such weakly convex shapes. 

A useful idea is that among maps between different shapes of a polytope, by rescaling the target shape via a homothety, we can find a map where at least one pair of distinct points $x, y$ has distance unchanged under the map and all other pairs of distinct points have distance unchanged or decreased, i.e a weak compression.  Simple examples illustrate that for weak compressions, pairs of distinct points whose distance apart is unchanged by the map, do not necessarily lie on a single edge of the polytope. However, for simplices, we establish that such pairs can always be chosen so that at least one of $x, y$ is a vertex. Similarly pairs of points where the distance is decreased by a maximum amount can be chosen so that at least one is a vertex. 

Given any two shapes of a polytope, by sufficiently rescaling one of the shapes, we obtain a compression map between the shapes. So, the scale of the shapes is crucial to the existence of a compression map. 

There is a vast literature on polytope realization as a convex set; we mention \cite{alex, grum, schrij, ziegler1}

We thank Simon Brendle for asking us for conditions that guarantee the existence of compression maps between polytopes. 

\section{Preliminaries}
 
A {\it shape} $P$ of a combinatorial polytope $\Pi$ is the convex hull of a finite set of points $V$ in $\mathbb{R}^d$, so that the elements of $V$ are the extreme points or vertices of $P$. Moreover, the face structure of $P$ is the same as that for $\Pi$. By this we mean that there is a bijection between the vertices of $\Pi$ and of $P$ which induces a bijection between the faces of $\Pi$ and the faces of $P$. Here, each face of $P$ is associated with its vertex set. 

In this paper, we restrict attention to simple polytopes. So each vertex belongs to exactly $d$ facets and in a shape $P$, each vertex belongs to $d$ hyperplanes defining facets of $P$. 

A shape $P$ of a polytope $\Pi$ is $d$-dimensional if among the vertex set, there are subsets of size $d+1$ but no larger, which are affinely independent. We will always require a shape of a $d$-dimensional combinatorial polytope to be non-degenerate, i.e., to have the same dimension $d$ as the polytope.

We identify two shapes of $\Pi$ if they differ by an isometry of $\mathbb{R}^d$. $\mathcal P$ is then the space of the shapes of a fixed combinatorial polytope $\Pi$ in $\mathbb{R}^d$ up to identification by an isometry.  Note that a shape of a polytope can also be characterised as a finite intersection of closed affine half-spaces. 

To enlarge the space of shapes of a polytope, we will work in the configuration space $\mathcal N$ of $N$ points in $\mathbb{R}^{d}$. $\mathcal N$ is the collections of $N$ points, which are not necessarily distinct, under the equivalence relation of ambient isometry. The space of shapes $\mathcal P$ of a polytope with $N$ vertices naturally embeds in $\mathcal N$. The larger space of weakly convex shapes of $\Pi$ is used to complete the projective space of shapes equipped with the compression metric. 
(A weakly convex shape is one where adjacent facets can have parallel normal vectors).

A $d$-simplex is the convex hull of $d+1$ affinely independent points in $\mathbb{R}^d$. An important way of triangulating polytopes without introducing extra vertices is  {\it regular} triangulations, which were introduced by Gelfand, Kapranov and Zelevinsky \cite{gelfand}. A triangulation of a realization $P$ of a $d$-dimensional polytope $\Pi$ is a decomposition into convex $d$-simplices, so that the union of all the simplices is $P$ and the intersection of any pair of simplices is empty or a face of each of the simplices. In \cite{gelfand} triangulations are constructed by lifting the vertices of $P$ to appropriate heights in $\mathbb{R}^{d+1}$. Then for a generic choice of heights, the lifted vertices have a convex hull whose lower boundary (in terms of the vertical direction in $\mathbb{R}^{d+1}$) projects to a triangulation of $P$, which is called regular. 

A second type of triangulation of a polytope is called {\it derived} or {\it barycentric}. Given a shape $P$ of a polytope of dimension $d$ and a face $F$ of $P$, we can define the barycentre $b(F)$ of $F$ as the point in the interior of $F$ which is the average of the vertices $v_1, \dots v_n$ of $F$. So, $b(F)= {1 \over n} \sum_{ 1 \le i \le n}  v_i$. Then the vertices of the barycentric triangulation $B(P)$ are all the barycentres of all the faces of $P$. The $d$-simplices of $B(P)$ are then the barycentres $b(F_0), b(F_1), \dots b(F_d)$ of a maximal chain of faces $F_0 \subset F_1 \dots \subset F_d$.

\section{Definitions and Fundamental Properties}

 Let $P, Q \in \mathcal P$ be two shapes of a given combinatorial polytope $\Pi$. Consider the barycentric triangulations $B(P), B(Q)$ of $P, Q$. There is a uniquely defined piecewise linear map $f:P \to Q$ which is a linear isomorphism restricted to each pair of simplices of $B(P), B(Q)$, induced by the correspondence between the vertices of $P,Q$. Note the fact that $P,Q$ are realizations of the same combinatorial polytope $\Pi$ implies there is a unique matching of their vertices with those of $\Pi$ and hence between their vertices. 
 
 We now describe some properties of $f$. Let $b(F)= {1 \over n} \sum_{ 1 \le i \le n}  v_i$  be a vertex of $B(P)$, where $F$ is a face with vertices $v_1, \dots v_n$. Then there is a corresponding vertex $f(b(F))$ of $B(Q)$. In fact $f(b(F)=b(G)$ where $G$ is the face of $Q$ with vertices $f(v_1), \dots f(v_n)$. The reason is that $f(b(F))={1 \over n} \sum_{ 1 \le i \le n}  f(v_i) = b(G)$.

 We will use $||u-v||$ to denote the distance between points $u,v \in \mathbb{R}^d$.
 
 \begin{Def} We say that $P$ compresses to $Q$ if $||f(x)-f(y)|| < ||x-y||$ for $x,y \in P$. We say that there is a weak compression from $P$ to $Q$ if $||f(x)-f(y)|| \leq ||x-y||$ for $x,y \in P$.
\end{Def}

\begin{Prop}\label{comptriang}
$P$ compresses (respectively weakly compresses) to $Q$ if and only if the map $f:P \to Q$ is a compression (respectively a weak compression) between corresponding simplices of the triangulations $B(P),B(Q)$ of $P,Q$ respectively. 

\end{Prop}

\begin{proof}
Clearly if $f$ is a compression (respectively weak compression), then $f$ restricted to a simplex is also a compression (respectively weak compression). So it suffices to prove the converse. Consider a pair of points $x,y \in P$. Let $\gamma$ be the straight line segment from $x$ to $y$. Then $\gamma$ lies in the convex set $P$ and passes through simplices $\sigma_1, \sigma_2, \dots \sigma_k$. Let $\gamma_1, \gamma_2, \dots \gamma_k$ be the segments of $\gamma$ in these simplices. Clearly each $\gamma_i$ is either a closed interval or a point. For $\gamma$ is a shortest path in $P$ and the shortest path between any two points in a simplex of $P$ is a straight line segment inside the simplex. 

Next, consider $f(\gamma)$. Since for each segment $\gamma_i$, $f$ decreases length (respectively does not increase length), it follows that the length of $f(\gamma)$ is strictly less (respectively is not longer) than the length of $\gamma$. If $f(\gamma)$ is not straight, then its length larger than the distance from $f(x)$ to $f(y)$. So this completes the proof that $f$ is a compression (respectively a weak compression) as required. 

\end{proof}

\begin{Rem} Consider the construction of regular triangulations for polytopes. Suppose an isomorphism between the combinatorial polytopes $\Pi, \Pi'$ induces an isomorphism between the corresponding polytopes $P,P'$, which are convex hulls of vertices lifted from $\mathbb{R}^d$ to $\mathbb{R}^{d+1}$. Then there is an induced isomorphism between the lower boundaries and hence between the induced regular triangulations. So in this case we could use regular triangulations rather than barycentric subdivisions to define compression maps. 
\end{Rem}

The projective shape space $\mathcal{P} / \sim$ of a combinatorial polytope $\Pi$, is the quotient of the space of convex realizations $\mathcal{P}$ by the equivalence relation $\sim$ of homothety. $\mathcal{P} / \sim$ is the space of polytope shapes up to uniform scaling, capturing the intrinsic geometry of realizations, independent of absolute size.

So two shapes $P, Q \in \mathcal{P}$ are identified in the projective shape space $\mathcal{P} / \sim$, if there exists a positive scalar $\lambda > 0$ such that $P = \lambda Q$.
We will use the notation $[P]$ to denote the equivalence class of $P$ in $\mathcal{P} / \sim$.

\begin{Ex}
The space of shapes $\mathcal{P}$ of an $n$-gon $\Pi$, for $n \ge 3$ is the cone over the open $n$-simplex $\sigma_n$ embedded in $\mathbb{R}^{n+1}$ by $$\sigma_n = \{x_i: 0 < x_i < \pi, \Sigma_i x_i = (n-2)\pi\}$$ The projective space of shapes $\mathcal{P} / \sim$ is then $\sigma_n$. 

\end{Ex}

\begin{Def}
For the space of shapes of a combinatorial polytope $\mathcal{P}$, we define a relation $\preceq$ by $P \preceq Q$  if there exists a weak compression map $f: P \to Q$ for $P,Q \in \mathcal{P}$. 

\end{Def}

\begin{Prop} \label{partialorder}
The relation $\preceq$ defined by weak compression induces a partial order on the space $\mathcal{P}$ of shapes of a combinatorial polytope. 
\end{Prop}

\begin{proof} For a shape $P \in \mathcal{P}$, the identity map $\operatorname{id}: P \to P$ is trivially a weak compression. So $P \preceq P$.

 Suppose $P, Q, R \in \mathcal{P}$ with weak compression maps $f: P \to Q$ and $g: Q \to R$. Then, for any $x, y \in P$
$$ \|g(f(x)) - g(f(y))\| \leq \|f(x) - f(y)\| \leq \|x - y\| $$
which implies that $g \circ f: P \to R$ is also a weak compression. Thus, $P \preceq Q$ and $Q \preceq R$ implies $P \preceq R$, proving transitivity.

Finally assume $P, Q \in \mathcal{P}$ and the maps $f: P \to Q$ and $g: Q \to P$ are weak compressions. Clearly $g \circ f$ induces the identity map on vertices and hence is the identity map on $P$. But then it follows  that $g,f$ are both distance-preserving maps on all the corresponding pairs of simplices of $B(P), B(Q)$. For if any pair of distinct points in a simplex has distance decreased by $f$ or $g$, then this would also be true for the composition $g \circ f$, giving a contradiction. It then follows that the isometries between the corresponding pairs of simplices of $B(P), B(Q)$ glue together to form an isometry between $P,Q$. 

This completes the proof that $\preceq$ is a partial order. 
\end{proof}

One might expect that if a piecewise linear map between two shapes of a combinatorial polytope induces a compression on the set of edges, then it must necessarily be a compression map. This intuition does not hold, even in simple cases.

\begin{Ex}
The map from triangle $P$ to triangle $Q$ is not a compression map in spite of the fact that it shrinks vertex-vertex and vertex-edge distances, because there exists a pair of points which it pushes away from each other. Consider:

\begin{align*}
\mathbf{p}_1 &= (0, 1.692); & \mathbf{q}_1 &= (3.452, 3.519) \\
\mathbf{p}_2 &= (3.452, 0.527); & \mathbf{q}_2 &= (4.696, 2.078) \\
\mathbf{p}_3 &= (1.901, 0); & \mathbf{q}_3 &= (4.393, 1.692)
\end{align*}
\vspace{-.7cm}
{\begin{figure} [htbp]
 \begin{center}
\includegraphics[width=.5\textwidth,angle=0]{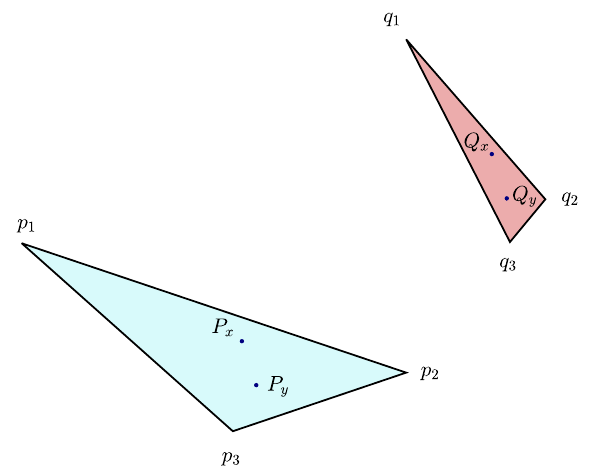}
\end{center}
\caption{A linear map between triangles inducing a compression on the edges is not necessarily a compression, since $ ||Q_x - Q_y || \leq  ||P_x - P_y|| $.}
\end{figure}}

Consider the maps from $ \sigma \subset \mathbb{R}^3 \to \mathbb{R}^2$, where $\sigma$ denotes the standard simplex:
$$ 
P = \begin{pmatrix}
0 & 3.452 & 1.901 \\
1.692 & 0.527 & 0
\end{pmatrix};
\quad 
Q = \begin{pmatrix}
3.452 & 4.696 & 4.393 \\
3.519 & 2.078 & 1.692
\end{pmatrix}
$$

Comparing vertex-vertex distances:
\begin{align*}
||\mathbf{p}_1- \mathbf{p}_2|| &= 3.64329; & ||\mathbf{q}_1- \mathbf{q}_2 || &= 1.90369 \\
||\mathbf{p}_1- \mathbf{p}_3|| &= 2.54493; & ||\mathbf{q}_1- \mathbf{q}_3|| &= 2.05509 \\
||\mathbf{p}_2- \mathbf{p}_3 || &= 1.63809; & ||\mathbf{q}_2 - \mathbf{q}_3|| &= 0.490719
\end{align*}

There exist a pair of points which the linear map in question pushes away from each other. Consider: 
$$
x = (0.34, 0.464, 0.196); \quad y = (0.149, 0.316, 0.535)
$$
$$
P_x = (1.97432, 0.819808); \quad P_y = (2.10787, 0.41864)
$$
$$
Q_x = (4.21365, 2.49228); \quad Q_y = (4.34854, 2.0862)
$$
$$
||P_x-P_y|| = 0.422811; \quad ||Q_x-Q_y|| = 0.427901
$$
Since, $$ || Q_x - Q_y || \leq || P_x - P_y || $$
the linear map between $P$ and $Q$ inducing a compression on the edges is not a compression.
\end{Ex}

\section{Weak Compressions and Compressions  of simplices as Orthogonal Projections}

In this section, we  study weak compression and compression maps between shapes of $d$-simplices. There is a unique linear map induced by locating two corresponding vertices at the origin and matching the other $d$ pairs of vertices. There is a more symmetric second way of constructing a linear map $\mathbb{R}^{d+1}\to\mathbb{R}^{d+1}$ by embedding the two simplices into the hyperplane with coordinate $x_{d+1}=1$. We discuss this below. 

Let $f:\mathbb{R}^d\to \mathbb{R}^d$ be a linear map. We denote $M\in M_d(R)$ as the standard matrix representation of $f$. $M$ is a weak compression, if $||Mx-My||\leq ||x-y||$ for all $x,y\in \mathbb{R}^d$. Note that we can also rewrite the weak compression condition as $$(x-y)^TM^TM(x-y) \le (x-y)^T(x-y)$$
A more succinct expression is $$I-M^TM \ge 0$$
Similarly the compression condition becomes $$I-M^TM > 0$$

\begin{Prop} \label{Proj} 
Orthogonal projections are weak compressions.
\end{Prop}
\begin{proof}  Let $M\in M_d(R)$ be the standard matrix representation of an orthogonal projection $p: \mathbb{R}^d \to \mathbb{R}^{d}$. The matrix $M$ is similar to a diagonal matrix $M'$ with $k$ ones and $d-k$ zeroes on the diagonal. Therefore, $||M'(x-y)|| \leq ||x-y||$, concluding the proof. 
\end{proof}

\begin{Def}
If $p$ is an orthogonal projection as in Proposition \ref{Proj}, we will call an affine line $L$ in $\mathbb{R}^d$ horizontal if the parallel line through the origin is orthogonal to the kernel of $p$. Alternatively, $L$ is horizontal if $p$ restricted to $L$ is an isometry to the line $p(L)$.

\end{Def}

\begin{Coro} \label{Proj1}
Suppose $P$ is a convex set in $\mathbb{R}^d$. Then the restriction of an orthogonal projection $p$ to $P$ is a weak compression. Moreover $p$ is a compression on $P$ if and only if any horizontal line meets $P$ is a single point. 

\end{Coro}

\begin{proof}
The fact that $p$ restricted to $P$ is a weak compression follows immediately from Proposition \ref{Proj}. If this restriction is not a compression, then there are pairs of distinct points $x,y \in P$ so that $||p(x) - p(y)||=||x-y||$. This implies the line $L$ through $x,y$ is horizontal. So there is a horizontal line intersecting $P$ in a line segment. The converse is similar. 

\end{proof}

In \cite{zhang}, a square matrix $M$ is called a weak compression if it satisfies $||Mx-My|| \le ||x-y||$.
Theorem 5.9 in \cite{zhang} establishes that a real square matrix is a weak compression if and only if it can be embedded into an orthogonal matrix of twice the dimension.

\begin{Thm}\label{orthog} 
A $d \times d$ matrix $M$ is a weak compression if and only if the matrix
\[ U=
\left(\begin{array}{@{}cc@{}}
  \begin{matrix}
  \hspace{.1cm}M \\
  \end{matrix}
  & B \hspace{.1cm} \\
 \hspace{.1cm} A &
  \begin{matrix}
  C \hspace{.1cm} \\
  \end{matrix}
\end{array}\right)
\]
is orthogonal for suitable $d \times d$ matrices $A,B,C$.
\end{Thm}

\begin{proof}
Let $f:\mathbb{R}^d \to \mathbb{R}^{2d}$ be the inclusion map given by $$f(a_1, \dots a_n) =(a_1, \dots ,a_n, 0, \dots ,0)$$
If $M$ embeds as a block in an orthogonal matrix $U$, then clearly $$||Mx-My|| \le ||Uf(x)-Uf(y)||=||f(x)-f(y)||=||x-y|| $$
Hence $M$ is a weak compression map. 

Let $g:\mathbb{R}^d \to \mathbb{R}^{2d}$ be the inclusion map given by $$g(a_1, \dots a_n) =(0, \dots ,0, a_1, \dots ,a_n)$$ To construct $U$ it clearly suffices to show we can find vectors $v_i$, $ 1 \le i \le d$ in $\mathbb{R}^{d}$ so that $f(x_i)+g(v_i)$ forms an orthonormal set. 

Now, the compression condition is $I-M^TM \ge 0$. Let $A$ be a real symmetric $d \times d$ nonnegative matrix such that $A^2=I-M^TM$. Then it is easy to check that the columns of $A$ give the required vectors $v_i$. It is now straightforward to extend the orthonormal set of size $d$ to an orthonormal basis for $\mathbb{R}^{2d}$ giving the matrices $B,C$.
\end{proof}

\begin{Coro} \label{ext}
A weak compression $f: \mathbb{R}^d \to \mathbb{R}^d$ can be extended to an orthogonal transformation ${\hat f}: \mathbb{R}^{2d} \to \mathbb{R}^{2d}$.
\end{Coro}
\begin{proof} 
This follows from Theorem \ref{orthog}
\end{proof}

\begin{Coro} \label{comp}
Given two shapes $P,Q$ of a $d$ simplex, there is a weak compression $P \to Q$ if and only if there is an isometric embedding of $P$ in $\mathbb{R}^{2d}$ so that the orthogonal projection to $\mathbb{R}^{d}$ takes $P$ onto $Q$.
\end{Coro}

\begin{proof}
One direction follows, since orthogonal projections are weak compressions by Proposition \ref{Proj}. The converse follows by Corollary \ref{ext}. For, if there is a weak compression from $P$ to $Q$, the extension to an orthogonal map on $\mathbb{R}^{2d}$ gives the desired isometric embedding realising the compression by an orthogonal projection. 
\end{proof}

\begin{Coro}
Given two shapes $P,Q$ of a $d$ simplex, there is a compression $P \to Q$ if and only the following is satisfied. 

\begin{itemize}
\item There is an isometric embedding of $P$ in $\mathbb{R}^{2d}$ so that the orthogonal projection to $\mathbb{R}^{d}$ takes $P$ onto $Q$.

\item The embedding of $P$ in $\mathbb{R}^{2d}$ intersects any horizontal line in at most one point. 

\end{itemize}

\end{Coro}

\begin{proof}
This follows using Corollary \ref{Proj1}

\end{proof}

\begin{Prop}
Suppose $f:P \to Q$ is a weak compression which is not a compression between two shapes for a $d$-simplex. Then
there is a vertex $x$ and a distinct point $y$ so $||f(x)-f(y)||=||x-y||$.
If $f$ is either a compression or a weak compression, there is a pair of points $x,y$ for which $\frac{||f(x)-f(y)||}{||x-y||}$ achieves a maximum value for $x,y \in P$ where at least one of $x,y$ is a vertex. 

\end{Prop}

\begin{proof}

Suppose that $x,y$ is a pair of points for which $||f(x)-f(y)||=||x-y||$. Notice that since $f$ is a linear map, any other pair of points $x',y'$ for which $x-y$ and $x'-y'$ are linearly dependent vectors satisfies the same value for $\frac{||f(x')-f(y')||}{||x'-y'||}$. But then we can choose $x'$ as a vertex and $y'$ a point in the interior of the simplex so that $x'-y'$ is a real multiple of $x-y$. Then clearly $||f(x')-f(y')||=||x'-y'||$.

Consider any two points $x,y$ for which $\frac{||f(x)-f(y)||}{||x-y||}$ achieves a maximum value. By the same argument as the previous paragraph, we can choose a vertex $x'$ and a distinct point $y$ for which $\frac{||f(x')-f(y')||}{||x'-y'||}$ achieves a maximum value. 

\end{proof}

\section{Weak Compressions and Compressions for some Polytopes}

In this section, we extend the characterization of compressions and weak compressions between shapes of simplices as orthogonal projections to the case of shapes of a special class of polytopes. 

\begin{Def} (pleated embedding).
Suppose that $\Pi$ is a $d$-dimensional combinatorial polytope with simplicial facets and with a triangulation $\mathcal T$ with vertices those of $\Pi$. Then a pleated realization of $\Pi$ is a collection of embedded convex $d$-simplices in $\mathbb{R}^{d'}$, for $d' > d$. The $d$-simplices intersect in faces corresponding to the triangulation $\mathcal T$. At each $(d-1)-$dimensional face shared by simplices $\sigma, \sigma'$, the angle between $\sigma, \sigma'$ is not necessarily $\pi$. Equivalently, $\sigma, \sigma'$ are not necessarily contained in the same $d$-dimensional affine subspace of $\mathbb{R}^{d'}$. Suppose that $\sigma_1, \dots ,\sigma_k$ are all the $d$-simplices of $\mathcal T$ that share a common codimension $2$ face $\tau$. We also require that the sum of the dihedral angles between the facets of $\sigma_1, \dots ,\sigma_k$  at $\tau$ is strictly less than $\pi$.
\end{Def} 

\begin{Ex}
A quadrilateral has two different triangulations without additional vertices, given by a choice of a diagonal. Choose one of these triangulations. If the quadrilateral is embedded in $\mathbb{R}^3$ so that the two triangles are affinely embedded and meet at an angle different from $\pi$ along the diagonal, the quadrilateral is no longer convex, but is pleated.
Note that we also require that the sum of the angles of the triangles at each of the two vertices at the diagonal is strictly less than $\pi$. 
\end{Ex}

\begin{Def}
Let $\mathcal{C}$ be the class of polytopes with simplicial facets that can be triangulated by successive truncations at vertices. 
\end{Def}

\begin{Def}
The face-pairing graph $\Gamma$ of a triangulation $\mathcal T$ of a $d$-dimensional polytope has vertices associated with each $d$-simplex of $\mathcal T$ and an edge connecting two such vertices if the associated simplices share a $(d-1)$-face. 

\end{Def}

\begin{Rem}
We will be interested in the special case where the  triangulation $\mathcal T$ of a polytope has a face-pairing graph $\Gamma$ which is a tree. The polytopes of this type are exactly those in the class $\mathcal C$. There are then two metrics induced on an embedded pleated realization of the polytope -- the extrinsic one from the ambient Euclidean space and the intrinsic one from the metrics on the simplices. The latter metric will be used to describe a shape of the polytope. 
\end{Rem}

\begin{Thm} \label{pleated}
Suppose that we are given two shapes $P, Q$ of a $d$-polytope $\Pi \in \mathcal{C}$ with a triangulation $\mathcal{T}$ with $t$ $d$-simplices. Assume the map $f:P \to Q$ defined using $\mathcal{T}$ is a weak compression or a compression. Then $f$ can be represented by orthogonal projection from a pleated embedding $p$ of $P$ into $\mathbb{R}^{d(t+1)}$ to the convex embedding $Q$ of $\Pi$ in $\mathbb{R}^d$. Here the intrinsic metric on the image of $p$ is isometric to $P$.  Conversely, any orthogonal projection from a pleated embedding of $P$ in $\mathbb{R}^{d+k}$ to a convex embedding of $Q$ in $\mathbb{R}^d$ is a weak compression map.
\end{Thm}

\begin{proof}
By Corollary \ref{comp}, if $\mathcal T$ is a single simplex, then the result follows. Let $P, Q$ be the two shapes of $\Pi$, where there is a weak compression or a compression map $P \to Q$. It is sufficient to glue together lifts of shapes of the simplices in the triangulation of $P$ to $\mathbb{R}^{2d}$ given by Corollary \ref{comp}. If the face-pairing graph $\Gamma$ for the simplicial decomposition of $\Pi$ is a tree with $t$ vertices, then we will show that such a gluing of lifts can be achieved by $t-1$ applications of Corollary \ref{comp}. This will give the required pleated embedding in $\mathbb{R}^{d(t+1)}$. 

The proof will be by induction on $t$. So suppose we have already constructed such a pleated embedding for the first $t-1$ simplices of $\mathcal T$. The last simplex has one vertex opposite a face which is in the previous set of simplices and so has already been embedded in $\mathbb{R}^{dt}$. So it suffices to show that if we allow an extra factor of $\mathbb{R}^{d}$ for the position of the last vertex, then the required pleated embedding can be extended over the final vertex. But this follows by the same method as Corollary \ref{comp}.
\end{proof}

\begin{Thm}\label{sequence}
Let $Q$ be a shape of a polytope $\Pi$ in the class $\mathcal{C}$, where $\Pi$ has simplicial facets and triangulation $\mathcal{T}$. Then $\Pi$ admits a sequence of pleated embeddings $P_{k}$, for $1 \le k \le dt$, respectively into spaces 
$
\mathbb{R}^{d+1}, \mathbb{R}^{d+2}, \dots, \mathbb{R}^{d(t+1)}
$
such that at each step, the embedding $P_{k}$ yields a compression map or weak compression map via orthogonal projection to the realization $Q$. Moreover each embedding $P_{k}$ in the sequence projects to all the previous embeddings $P_{i}$ for $1 \le i \le k-1$ via orthogonal projection. 

Finally, any compression or weak compression between two realizations $P,Q$ of $\Pi$ in $\mathbb{R}^d$ can be constructed as an orthogonal projection from a pleated  embedding $P_k$ of $P$ in $\mathbb{R}^{d+k}$ for suitable $k =dt$. Note that the intrinsic metric of $P_k$ can be chosen to be isometric to the required shape $P$ for each $k$. 
\end{Thm}

\begin{proof}

By Theorem \ref{pleated}, the final assertion of the Theorem follows. So given two shapes $P,Q$ of $\Pi$ in $\mathcal C$ with a compression or weak compression between them, we can realize $P$ by a pleated embedding in $\mathbb{R}^{t(d+1)}$.

But now we can orthogonally project from $\mathbb{R}^{t(d+1)}$ to $\mathbb{R}^{(d+k)}$ for $1 \le k<td$ and achieve the required sequence of pleated embeddings of $\Pi$ using the images of $P$.

\end{proof}

\begin{Rem}
Theorem \ref{sequence} can be used to build all compressions and weak compressions by successively lifting one dimension at a time. Given a shape $Q$ of $\Pi$ in $\mathbb{R}^{d}$, there is a simple method to construct all liftings of each $d$-simplex to $\mathbb{R}^{d+1}$. 
Namely some vertex can be located at the origin of $\mathbb{R}^{d+1}$ and then all the other vertices have the same coordinates in $\mathbb{R}^{d}$ and an arbitrary height in the extra coordinate. This can be iterated until all possible compressions and weak compressions are constructed.

\end{Rem}

\section{Compression Metric and Perturbations of Shapes of Simplices}

We embed the $d$-simplices ${P} = \mbox{conv}(\mathbf{p}_1, \ldots, \mathbf{p}_{d + 1})$ and ${Q} = \mbox{conv}(\mathbf{q}_1, \ldots, \mathbf{q}_{d + 1})$ in the hyperplane $x_{d + 1} = 1$ of $\mathbb{R}^{d+1}$, so their coordinates form the columns of the matrices  
$$
\mathbf{P} = \begin{pmatrix}
\mathbf{p}_1 & \hdots & \mathbf{p}_{d + 1}\\
1 & \hdots & 1
\end{pmatrix}, \quad
\mathbf{Q} = \begin{pmatrix}
\mathbf{q}_1 & \hdots & \mathbf{q}_{d + 1}\\
1 & \hdots & 1
\end{pmatrix}
$$

The affine linear transformation mapping ${P}$ to ${Q}$ is given by $\mathbf{A} = \mathbf{Q} \mathbf{P}^{-1}$. By removing the last row and column of $\mathbf{A}$, we obtain a reduced $d \times d$ matrix denoted as $\overline{\mathbf{A}}$.

\begin{Lem} \label{contrac}
A linear map $\mathbf{A}$ is a weak compression if and only it is a contraction, meaning $\|\mathbf{A}\| \leq 1$. It is a compression if and only if $\|\mathbf{A}\| < 1$.
\end{Lem}

\begin{proof}
$\mathbf{A}$ is a weak compression if $\|\mathbf{A} \mathbf{x}\| \leq \|\mathbf{x}\|$ for all $\mathbf{x} \in \mathbb{R}^d$. Taking the supremum over unit vectors gives $\|\mathbf{A}\| \leq 1$. Conversely, if $\|\mathbf{A}\| \leq 1$, then for any $\mathbf{x} = r \mathbf{y}$ with $\|\mathbf{y}\| = 1$  
$$
\|\mathbf{A} \mathbf{x}\| = r \|\mathbf{A} \mathbf{y}\| \leq r = \|\mathbf{x}\|
$$
 The compression case follows analogously.
\end{proof}

\begin{Prop}  
The set of linear weak compressions forms a convex subset of the space of linear maps from $\mathbb{R}^d$ to itself.
\end{Prop}

\begin{proof}
Suppose $A$ is a weak compression and let $\lambda \in [0,1]$.
Define 
$
B =\lambda A +(1 - \lambda)\,I,
$
where $I$ is the identity map. Then, for any \(\mathbf{x} \in \mathbb{R}^d\), we have
$$
\|B\mathbf{x}\| \;=\; \|\lambda A \mathbf{x} + (1-\lambda)\mathbf{x}\|
\;\le\; \lambda \|A \mathbf{x}\| + (1-\lambda)\|\mathbf{x}\| \;\le\; \|\mathbf{x}\|.
$$
Thus $B$ is also a compression.
\end{proof}

\begin{Prop}\label{evallesseq1}
The affine map $\mathbf{A} = \mathbf{Q} \mathbf{P}^{-1}$ sending a $d$-simplex ${P}$ to a $d$-simplex ${Q}$ is a weak compression if and only if all eigenvalues of $\overline{\mathbf{A}}^T \overline{\mathbf{A}}$ are at most one.
\end{Prop} 
\begin{proof}

If $\mathbf{A}$ is a compression, then for any $\mathbf{x}, \mathbf{y} \in {P}$,
$$
(\mathbf{A} \mathbf{x} - \mathbf{A} \mathbf{y})^T (\mathbf{A} \mathbf{x} - \mathbf{A} \mathbf{y}) \leq (\mathbf{x} - \mathbf{y})^T (\mathbf{x} - \mathbf{y}).
$$
Using $\mathbf{A} = \mathbf{Q} \mathbf{P}^{-1}$ and restricting to the affine subspace, we obtain
$$
(\mathbf{x} - \mathbf{y})^T \overline{\mathbf{A}}^T \overline{\mathbf{A}} (\mathbf{x} - \mathbf{y}) \leq (\mathbf{x} - \mathbf{y})^T (\mathbf{x} - \mathbf{y}).
$$
Since this holds for all $\mathbf{x}, \mathbf{y}$, it follows that all eigenvalues of $\overline{\mathbf{A}}^T \overline{\mathbf{A}}$ are at most one.

Conversely, if all eigenvalues of $\overline{\mathbf{A}}^T \overline{\mathbf{A}}$ satisfy $\lambda_i \leq 1$, then for any vector $\mathbf{v}$,
$$
\mathbf{v}^T \overline{\mathbf{A}}^T \overline{\mathbf{A}} \mathbf{v} 
= \sum_{i=1}^{n} \lambda_i c_i^2 \leq \sum_{i=1}^{n} c_i^2 = \mathbf{v}^T \mathbf{v}
$$
Applying this to $\mathbf{v} = \mathbf{x} - \mathbf{y}$ for any $\mathbf{x}, \mathbf{y} \in {P}$ shows that distances do not increase under $\mathbf{A}$, concluding the proof.
\end{proof}

\subsection{Perturbations of shapes of simplices}

\begin{Prop}
    An infinitesimal perturbation of the shape $\mathbf{P}$ of a simplex in direction $\mathbf{V}$ is a weak compression if and only if $\overline{\mathbf{V}\mathbf{P}^{-1}} +(\overline{\mathbf{V}\mathbf{P}^{-1}})^T$ is negative semidefinite.

\end{Prop}

\begin{proof}
Let $\mathbf{P}$ and $\mathbf{Q} = \mathbf{P} + t \mathbf{V}$ be the vertex matrices of the simplices $P$ and $Q$, for $t \in \mathbb{R}$. (We abuse notation by suppressing the dependence of $Q, \mathbf{Q}$ on $t$.) The affine map sending $P$ to $Q$ is given by

\begin{equation*}
\mathbf{A} = \mathbf{Q} \mathbf{P}^{-1} = (\mathbf{P} + t \mathbf{V}) \mathbf{P}^{-1} = \mathbf{I} + t \mathbf{V} \mathbf{P}^{-1}
\end{equation*}

By Proposition \ref{evallesseq1} the eigenvalues of $\overline{\mathbf{A}}^T \overline{\mathbf{A}}$ must be less than or equal to one, ensuring that the map does not expand any distances. For an infinitesimal perturbation to be a compression, the eigenvalues of the derivative $\frac{d}{dt} \overline{\mathbf{A}}^T \overline{\mathbf{A}} \big|_{t = 0}$ must be less than or equal to zero. By the product rule, we have

\begin{equation*}
(\overline{\mathbf{A}}^T \overline{\mathbf{A}})' = \overline{\mathbf{A}}^T \overline{\mathbf{A}}' + (\overline{\mathbf{A}}^T)' \overline{\mathbf{A}},
\end{equation*}

from which it follows

\begin{align*}
\frac{d}{dt} \overline{\mathbf{A}}^T \overline{\mathbf{A}} \big|_{t = 0} &= \overline{\mathbf{A}}' + (\overline{\mathbf{A}}^T)'\\
\end{align*}

Thus, the weak compression property is preserved if and only if $\overline{\mathbf{A}}' + (\overline{\mathbf{A}}^T)'$ is negative semidefinite. Since $\overline{\mathbf{A}}=\overline{\mathbf{I}} +t {\overline{\mathbf{Q} \mathbf{P}^{-1}}}$ the result follows. 
\end{proof}

Next, we note that a perturbation that does not increase edge lengths can still induce a first-order weak compression, shrinking distances from a vertex to an interior point of an edge.

\begin{Ex}
Consider a rhombus with vertices  
$$
\mathbf{p}_1 = (-a,0), \quad \mathbf{p}_2 = (0,b), \quad \mathbf{p}_3 = (a,0), \quad \mathbf{p}_4 = (0,-b),
$$
where all edge lengths are equal. Define an infinitesimal perturbation:  
$$
\mathbf{v}_1 = (0, -\epsilon), \quad \mathbf{v}_2 = (\epsilon, 0), \quad \mathbf{v}_3 = (0, \epsilon), \quad \mathbf{v}_4 = (-\epsilon, 0).
$$
This moves horizontal vertices vertically and vertical ones horizontally, preserving edge lengths, since  
$$
\frac{d}{dt} \|\mathbf{p}_i + t\mathbf{v}_i - (\mathbf{p}_j + t\mathbf{v}_j)\| \Big|_{t=0} = 0.
$$
Now, consider the midpoint of the edge $\mathbf{p}_3 \mathbf{p}_4$:  
$$
\mathbf{m} = \left( \frac{a}{2}, -\frac{b}{2} \right), \quad
\mathbf{m}' = \left( \frac{a}{2}, -\frac{b}{2} + \frac{t\epsilon}{2} \right).
$$
The distance from $\mathbf{p}_1$ to $\mathbf{m}$ is  
$$
\|\mathbf{p}_1 - \mathbf{m}\| = \sqrt{\left(-a - \frac{a}{2}\right)^2 + \left(\frac{b}{2}\right)^2}.
$$
After perturbation, expanding to first order in $t$,  
$$
\|\mathbf{p}_1 + t \mathbf{v}_1 - \mathbf{m}'\| = \|\mathbf{p}_1 - \mathbf{m}\| - \frac{t\epsilon}{2}.
$$
Thus, the vertex moves closer to the interior point while edge lengths remain unchanged, illustrating that preserving edge lengths does not necessarily prevent compression in other distances.
\end{Ex}

\subsection{Compression Metric}

We define a metric on the projective shape space $\mathcal{P}/\sim$ of a simplex as follows. Let $\alpha_{\min}, \dots, \alpha_{\max}$ be the eigenvalues of $\overline{\mathbf{A}}^T \overline{\mathbf{A}}$. Define $\Delta: \mathcal{P}/\sim \times \mathcal{P}/\sim \to \mathbb{R}$ by  
$$
\Delta([P], [Q]) = \ln \left(\frac{\alpha_{\max}}{\alpha_{\min}}\right)
$$
Here $[P]$ denotes an equivalence class of shapes under the relation of homothety, containing the shape $P$.

\begin{Prop}\label{metric}
The mapping $\Delta: \mathcal{P}/\sim \times \mathcal{P}/\sim \to \mathbb{R}$ defines a metric.
\end{Prop}

\begin{proof} \hfill
\begin{enumerate}

	\item If $P \sim P'$, then $P = \lambda P'$ for $\lambda \in \mathbb{R^+}$. Hence it follows immediately that if $P \sim P'$ and $Q \sim Q'$, then $\Delta(P,Q)=\Delta(P',Q')$. So $\Delta$ is well-defined on $\mathcal{P}/\sim$.
	
    \item $\Delta([P],[P])=0$, as the eigenvalues of the identity matrix are unity.
    
  \item $\Delta([P],[Q]) > 0$ whenever $[P] \neq [Q]$, since $\alpha_{\min} < \alpha_{\max}$, assuming that $P$ and $Q$ are not equivalent under a rigid body motion or a homothety.

    \item symmetry is satisfied since
    \begin{equation*}
        \Delta([P], [Q]) = \ln \left(\frac{\alpha_{\max}}{\alpha_{\min}}\right) = \ln \left(\frac{\alpha_{\min}^{-1}}{\alpha_{\max}^{-1}}\right) = \Delta([Q], [P]),
    \end{equation*}
    as the minimum and maximum eigenvalues of $\overline{\mathbf{A}^{-1}}^T \overline{\mathbf{A}^{-1}}$ are $\alpha_{\max}^{-1}$ and $\alpha_{\min}^{-1}$.
    
     \item The triangle inequality holds:
\begin{equation*}
    \Delta([P], [R]) \leq \Delta([P], [Q]) + \Delta([Q], [R]).
\end{equation*}

    Let $\mathbf{A}: P \mapsto Q$, $\mathbf{B}: Q \mapsto R$, and $\mathbf{C}: P \mapsto R$. The symmetric matrices $\overline{\mathbf{A}}^T \overline{\mathbf{A}}$, $\overline{\mathbf{B}}^T \overline{\mathbf{B}}$, and $\overline{\mathbf{C}}^T \overline{\mathbf{C}}$ have minimum amd maximum eigenvalues $\alpha_{\min}, \alpha_{\max}$, $\beta_{\min}, \beta_{\max}$, and $\gamma_{\min}, \gamma_{\max}$, respectively.  

    The eigenvector $\mathbf{c}_{\min}$ associated with $\gamma_{\min}$ satisfies  
    \begin{align*}
        \gamma_{\min} &= \mathbf{c}_{\min}^T \overline{\mathbf{C}}^T \overline{\mathbf{C}} \mathbf{c}_{\min}\\
        &= \mathbf{c}_{\min}^T \overline{\mathbf{A}}^T \overline{\mathbf{B}}^T \overline{\mathbf{B}} \overline{\mathbf{A}} \mathbf{c}_{\min}\\
        &\geq \mathbf{c}_{\min}^T \overline{\mathbf{A}}^T \beta_{\min} I_n \overline{\mathbf{A}} \mathbf{c}_{\min}\\
        &\geq \alpha_{\min} \beta_{\min}.
    \end{align*}
    Similarly, $\gamma_{\max} \leq \alpha_{\max} \beta_{\max}$. It follows that  
    \begin{align*}
        \Delta([P], [R]) &= \ln \left(\frac{\gamma_{\max}}{\gamma_{\min}}\right)\\
        &\leq \ln \left(\frac{\alpha_{\max} \beta_{\max}}{\alpha_{\min} \beta_{\min}}\right)\\
        &= \ln \left(\frac{\alpha_{\max}}{\alpha_{\min}}\right) + \ln \left(\frac{\beta_{\max}}{\beta_{\min}}\right)\\
        &= \Delta([P], [Q]) + \Delta([Q], [R]),
    \end{align*}
    proving the triangle inequality.
\end{enumerate}
\end{proof}

\begin{Thm}\label{complete}
The projective space of shapes $\mathcal{P}/\sim$ of a simplex, equipped with the compression metric
$$
\Delta([P], [Q]) = \ln \left(\frac{\alpha_{\max}}{\alpha_{\min}}\right)
$$
where $\alpha_{\max}$ and $\alpha_{\min}$ are the maximum and minimum eigenvalues of the symmetric matrix $\overline{\mathbf{A}}^T \overline{\mathbf{A}}$, is complete.
\end{Thm}

\begin{proof}
Let $\{[P_n]\} \subset \mathcal{P}/\sim$ be a Cauchy sequence. Then, given $\epsilon > 0$, there exists $N$ such that for all $m, n \geq N$
$$
\Delta([P_n], [P_m]) = \ln \left(\frac{\alpha_{\max}}{\alpha_{\min}}\right) < \epsilon
$$
Since the logarithm is continuous and strictly increasing, this implies
$$
\lim_{m,n \to \infty} \frac{\alpha_{\max}}{\alpha_{\min}} = 1
$$

Let $\mathbf{A}_{nm}$ be the transformation matrix between shapes $[P_n]$ and $[P_m]$, defined as
$$
\mathbf{A}_{nm} = \mathbf{Q}_m \mathbf{P}_n^{-1}
$$
where $\mathbf{Q}_m$ and $\mathbf{P}_n$ are representations of shapes $P_m$ and $P_n$. The eigenvalues of $\overline{\mathbf{A}}_{nm}^T \overline{\mathbf{A}}_{nm}$ converge uniformly to $\lambda$,  where $0 < \lambda \le 1$. By the spectral theorem for symmetric matrices,
$$
\overline{\mathbf{A}}_{nm}^T \overline{\mathbf{A}}_{nm} = \mathbf{Q}_{nm} \mathbf{D}_{nm} \mathbf{Q}_{nm}^T
$$
where $\mathbf{D}_{nm}$ is a diagonal matrix with entries $\alpha_1, \alpha_2, \dots, \alpha_n$. 

Since $\lim_{m,n \to \infty} \alpha_i = \lambda$ for all $i$, it follows that
$$
\lim_{m,n \to \infty} \mathbf{D}_{nm} = \lambda \mathbf{I}
$$
Therefore
$$
\lim_{m,n \to \infty} \overline{\mathbf{A}}_{nm}^T \overline{\mathbf{A}}_{nm} = \lim_{m,n \to \infty} \mathbf{Q}_{nm} \mathbf{D}_{nm} \mathbf{Q}_{nm}^T = \lambda\mathbf{Q}_* \mathbf{I} \mathbf{Q}_*^T = \lambda \mathbf{I},
$$
where $\mathbf{Q}_*$ is the limit of $\mathbf{Q}_{nm}$.

The convergence of $\overline{\mathbf{A}}_{nm}^T \overline{\mathbf{A}}_{nm}$ to $\lambda \mathbf{I}$ ensures that the sequence of reduced transformation matrices $\overline{\mathbf{A}}_n$ converges to a well-defined limit $\overline{\mathbf{A}}_*$. Consequently, the projective shape sequence $\{[P_n]\}$ converges to a limiting shape $[P_*]$ in $\mathcal{P}/\sim$, concluding the proof.
\end{proof}

\section{Compression metric for polytopes}

We can extend the results of the previous section to general polytopes. 

Consider the projective shape space $\mathcal{P}/\sim$ of a polytope. Recall that a shape of the polytope $P$ is obtained from the barycentric subdivision $B(P)$ by a realization as a convex shape in $\mathbb{R}^d$ where each simplex in $B(P)$ has an induced shape given by the vertices being  located at the barycentres of faces of $P$.

\begin{Def}
Let $\sigma_1, \dots \sigma_t$ be the barycentric simplex shapes for a realization $P$ and $\tau_1, \dots \tau_t$ be the corresponding shapes for a second realization $Q$ of the same polytope. We define the compression metric $\Delta$ on $\mathcal{P}/\sim$ as $$\Delta([P],[Q])=max_{1 \le i \le t}\Delta([\sigma_i],[\tau_i])$$

\end{Def}

\begin{Prop} \label{metricpoly}
$\Delta$ is a metric on the projective shape space of a polytope $\mathcal{P}/\sim$.

\end{Prop}

\begin{proof}
This follows easily from Proposition \ref{metric}. For it is straightforward to check the properties of a metric follow since $\Delta$ is the maximum of the distances between the corresponding shapes of the simplices of the barycentric subdivisions of two realizations of the polytope. 

The only property which does not follow immediately is non-degeneracy, i.e. that if $\Delta([P],[Q])=0$ then $[P]=[Q]$. To check this, note that 
$\Delta([P],[Q])=0$ implies $\Delta([\sigma_i],[\tau_i])=0$ for $1 \le i \le t$. But then there are homotheties so that $\sigma_i=\lambda_i \tau_i$, for $1 \le i \le t$. Since the shapes of $\sigma_i$ and $\tau_i$ must glue together to form convex realizations of the same polytope, it is easy to see that this implies all the homothety factors $\lambda_i$ must be the same. Hence there is a homothety between the realizations $P,Q$ and $[P]=[Q]$ as required 

\end{proof}

Recall that given a polytope $P$ with $N$ vertices, its shape space $\mathcal P$ embeds in the configuration space $\mathcal N$ of $N$ points in $\mathbb{R}^d$.

\begin{Def}\label{weak}
The space of weakly convex realizations of $P$ consists of convex subsets of $\mathbb{R}^d$ satisfying one or the other of the following conditions.

\begin{itemize}
 \item The extreme points of the convex subsets can be identified with all the vertices of $P$ and the face structure is that of $P$. These are the usual realizations of $P$.
 
 \item The extreme points of the convex subsets correspond to a proper subset of the vertices of $P$. Additionally, the other vertices of $P$ can be identified with non-extreme points in the boundary of the realization, so that each face of $P$ is realized as the convex hull of its vertices. 
\end{itemize}
\end{Def}

The simplest example of the second situation in Definition \ref{weak} is where two facets become part of the same codimension one affine subspace. In this case, the dihedral angle between these facets is $\pi$. 

\begin{Thm}
The projective shape space of a polytope has completion consisting of the projective space of weakly convex realizations of $P$. 

\end{Thm}

    \begin{proof}
    
    It suffices to show that any Cauchy sequence in the projective space of weakly convex realizations has a convergent subsequence. The argument follows the same approach as in Theorem \ref{complete}. 
    
    Note that for a Cauchy sequence of shapes, it is easy to see that the shapes of individual simplices of the barycentric subdivisions form Cauchy sequences. Therefore, as in Theorem \ref{complete}, these sequences have limiting shapes of the simplices in the projective shape space of a simplex. But now if we scale the simplices so that they fit together to form a convex shape for the polytope, the limit shapes must also fit together to form such a convex shape. See the proof of Proposition \ref{metricpoly}. Hence the Cauchy sequence does converge as required. 

    The key difference to Theorem \ref{complete} is that the dihedral angle between adjacent simplices in the limiting barycentric subdivision can be $\pi$. So the limiting shape can be a weakly convex realization as well as a convex realization of the polytope.

    \end{proof}

\bibliographystyle{amsplain}

\begin{thebibliography}{11} \rm
   
\bibitem{alex} Alexandrov, A.D., Convex polyhedra, Gosudarstv. Izdat. Tehn.-Teor. Lit., Moscow-
Leningrad 1950, Akademie-Verlag, Berlin 1958, Springer-Verlag, Berlin 2005



\bibitem{gelfand} I.M. Gelfand, M.M. Kapranov, A.V. Zelevinsky, Hyperdeterminants, Adv.
Math. 96 (1992), 226-263 

\bibitem{grum}Grünbaum, B., Convex polytopes, Interscience,Wiley, London 1967, 2nd ed., prepared
by V. Kaibel, V. Klee, G.M. Ziegler, Springer-Verlag, New York 2003

\bibitem{schrij} Schrijver, A., Combinatorial optimization. Polyhedra and efficiency A,B,C, Springer-
Verlag, Berlin 2003

\bibitem{zhang} Zhang, Fuzhen, Matrix theory: basic results and techniques. New York: Springer, 2011.
\bibitem{ziegler1} Ziegler, Günter M., Lectures on Polytopes, Graduate Texts in Mathematics, vol. 152, Berlin, New York: Springer-Verlag.

\end{thebibliography}

\end{document}